\documentclass[14pt,oneside]{extarticle}
\usepackage{setspace}
\usepackage[T1,T2A]{fontenc}
\usepackage[utf8]{inputenc}
\usepackage[english, russian]{babel}
\usepackage{titlesec}
\usepackage{titletoc}
\usepackage{fancyhdr}
\usepackage{indentfirst}
\usepackage{amsthm}
\usepackage{amssymb}
\usepackage{amsmath}
\usepackage{amsfonts}
\usepackage{verbatim}
\usepackage{cite}
\usepackage{enumitem}
\usepackage{ragged2e}
\usepackage[a4paper,headheight=17pt,left=20mm,right=20mm,bottom=20mm,top=20mm]{geometry}
\usepackage{stmaryrd}
\usepackage{tikz}
\usetikzlibrary{positioning}

\tikzset{sgplattice/.style={inner sep=1pt,norm/.style={red!50!blue},char/.style={blue!50!black},
  lin/.style={black!75}},cnj/.style={black!50,yshift=-2.5pt,left=-1pt of #1,scale=0.5,fill=white}}

\titleformat{\section}[block]{\bfseries}%
	{\bfseries\hspace{12.5mm} \thesection \ }{0pt}{}
\titleformat{\subsection}[block]{\normalfont\bfseries}%
	{\bfseries\hspace{12.5mm}\thesubsection\ }{0pt}{}

\newcommand{\sectionbreak}

\swapnumbers

\newtheoremstyle{theorem}
{0pt}
{0pt}
{\sl}
{\parindent}
{\bf}
{. }
{ }
{}

\theoremstyle{theorem}
\newtheorem{theorem}{Theorem}     
\newtheorem{lemma}{Lemma}          
\newtheorem{corollary}{Corollary}[theorem]   

\newtheoremstyle{definition}
{0pt}
{0pt}
{}
{\parindent}
{\bf}
{. }
{ }
{}
\theoremstyle{definition}
\newtheorem{definition}{Definition}  
\newtheorem{example}{Example}      

\renewenvironment{proof}[1][Proof]{{\bf #1. }}{\qed}

\usepackage{xcolor}

\usepackage{tikz}
\usetikzlibrary{positioning}

\tikzset{sgplattice/.style={inner sep=1pt,norm/.style={red!50!blue},char/.style={blue!50!black},
  lin/.style={black!50}},cnj/.style={black!50,yshift=-2.5pt,left=-1pt of #1,scale=0.5,fill=white}}

\makeatletter
\def\thmhead@plain#1#2#3{%
  \thmname{#1}\thmnumber{\@ifnotempty{#1}{ }\@upn{#2}}%
  \thmnote{ {\the\thm@notefont#3}}}
\let\thmhead\thmhead@plain

\def\swappedhead#1#2#3{%
  \thmname{#1}~\thmnumber{#2}
  \thmnote{ {\the\thm@notefont#3}}}
\let\swappedhead@plain=\swappedhead
\makeatother

\makeatletter
\renewcommand*\l@section{\@dottedtocline{0}{0em}{1.4em}}
\makeatother

\makeatletter
\renewcommand\@biblabel[1]{#1}
\makeatother

\makeatletter 
\renewcommand{\@tocrmarg}{2.55em plus1fil}
\makeatother 

\setenumerate{noitemsep,topsep=0pt,parsep=0pt,partopsep=0pt,leftmargin=0pt,itemindent=18.5mm}

\begin{document}

\abovedisplayskip=8pt plus 2pt minus 2pt
\belowdisplayskip=8pt plus 2pt minus 2pt

\begin{center}
  \Large{\textbf{On the intersection of $\mathfrak{F}$-maximal subgroups\\ of a finite group}}\footnote{The work was carried out with the financial support of the Belarusian Republican Foundation for Fundamental Research (BRFFR-RSF M, project F23RNFM-63).}\normalsize
  
  \hspace{5mm}
  
  Viachaslau I. Murashka, Yana A. Kuptsova
  
  \hspace{5mm}
  
  \small Faculty of Mathematics and Technologies of Programming,\\ Research Laboratory <<Mathematics of Hybrid Intelligence Systems>>,\\ Francisk Skorina Gomel State University, Gomel,  Belarus
  
  \small mvimath@yandex.by, kuptsovayana519@gmail.com
\end{center}

\textbf{Abstract:} We investigate the properties of the intersection $\mathrm{Int}_{\mathfrak{F}}(G)$ of all $\mathfrak{F}$-maximal subgroups of a finite group $G$ for a hereditary formation $\mathfrak{F}$ of finite groups. We prove that $\mathrm{Int}_{\mathfrak{F}}(G/\mathrm{Int}_{\mathfrak{F}}(G))\simeq 1$ holds for any finite group $G$ if and only if $\mathfrak{F}$ contains every group $G$ all of whose $\mathfrak{F}$-subgroups are $\mathfrak{F}$-subnormal. As corollaries we obtain the results of A.~N.~Skiba (2011), J.~C.~Beidleman and H.~Heineken (2011) about $\mathrm{Int}_{\mathfrak{F}}(G)$ for a hereditary saturated formation $\mathfrak{F}$.


\section*{Introduction and the main results}

All groups considered here are finite. The concept of the hypercenter as the final term of the upper central series plays an important role in the modern group theory. The development of this concept within the framework of the theory of formations was carried out in the works of R.~Baer \cite{Baer1959} for the class of all supersolvable groups, B.~Huppert \cite{B.Huppert1969} for local formations, L.~A.~Shemetkov for graduated formations \cite{Shemetkov1974} and appeared in its final form in the monograph of L.~A.~Shemetkov and A.~N.~Skiba \cite{ShemetkovSkiba1989}. Let $\mathfrak{F}$ be a formation. Recall (see \cite[p. 6]{Guo2015book} or \cite{ShemetkovSkiba1989}) that a chief factor $H/K$ of a group $G$ is called $\mathfrak{F}$-central provided that the semidirect product of $H/K$ with $G/C_G(H/K)$ belongs to $\mathfrak{F}$, where $G/C_G(H/K)$ acts on $H/K$ by conjugation. The largest normal subgroup of $G$ such that all chief factors of $G$ below it are $\mathfrak{F}$-central in $G$ is called the $\mathfrak{F}$-hypercenter (see \cite{Guo2015book} or \cite[p.142]{ShemetkovSkiba1989}) and is denoted by $\mathrm{Z}_{\mathfrak{F}}(G)$.

Before the mentioned works R.~Baer in 1966 \cite{Baer1966} proposed a different approach to define the $\mathfrak{e}$-hypercenter for a not necessary finite group and a group-theoretical property $\mathfrak{e}$. The $\mathfrak{e}$-hypercenter $\mathfrak{h}_\mathfrak{e}G$ of $G$ is the set of elements $g$ of $G$ such that $\langle g,E\rangle$ is a subgroup of some $\mathfrak{e}$-subgroup of $G$ for any $\mathfrak{e}$-subgroup $E$ of $G$ \cite[$\S 4$]{Baer1966}. From \cite{Baer1953} it follows that if $\mathfrak{e}$ is nilpotency, then $\mathfrak{h}_\mathfrak{e}G$ is the hypercenter $G$. Note that if $\mathfrak{e}$ is the property of a group to belong to a class $\mathfrak{F}$, then $\mathfrak{h}_\mathfrak{e}G$ is the intersection $\mathrm{Int}_{\mathfrak{F}}(G)$ of all $\mathfrak{F}$-maximal subgroups of $G$. Various properties of this subgroup were studied by A.~N.~Skiba \cite{Skiba2011, Skiba2012}, W.~Guo and A.~N.~Skiba \cite{WenbinSkiba2012}, J.~С.~Beidleman and H.~Heineken \cite{Beidleman2010}, A.~Lucchini and D.~Nemmi \cite{Lucchini2021}, V.~I.~Murashka and A.~F.~Vasil'ev \cite{Murashka2022}. One of the key properties of the $\mathfrak{F}$-hypercenter, in particular the hypercenter, is: $\mathrm{Z}_{\mathfrak{F}}(G/\mathrm{Z}_{\mathfrak{F}}(G))\simeq 1$ for any group $G$. An analogue of this property for $\mathrm{Int}_{\mathfrak{F}}(G)$ was obtained in \cite{Skiba2011, Skiba2012} when $\mathfrak{F}$ is a hereditary saturated formation. 

\begin{theorem}\label{theorem1}
There exists a hereditary solubly saturated formation $\mathfrak{F}$ and a group $G$ such that $\mathrm{Int}_{\mathfrak{F}}(G/\mathrm{Int}_{\mathfrak{F}}(G))\not \simeq 1$.
\end{theorem}

Therefore the problem to describe all hereditary formations $\mathfrak{F}$ such that $\mathrm{Int}_{\mathfrak{F}}(G/\mathrm{Int}_{\mathfrak{F}}(G))\simeq 1$ holds for any group $G$ seems natural. Our solution to this problem is closely related to the concept of an $\mathfrak{F}$-subnormal subgroup. Recall (see \cite[Definition 6.1.2]{BallesterBolinches2006}) that a subgroup $H$ of $G$ is called $\mathfrak{F}$-subnormal in $G$ (shortly $H$ $\mathfrak{F}$-sn $G$) if either $H=G$ or there exists a maximal chain of subgroups $H=H_0\subset H_1 \subset \ldots \subset H_n=G$ with $H_{i}/Core_{H_i}(H_{i-1})\in \mathfrak{F}$~for~any~$i=1,~\ldots,~n$.

\begin{theorem}\label{theorem2} Let $\mathfrak{F}$ be a non-empty hereditary formation. Then $\mathrm{Int}_{\mathfrak{F}}(G/\mathrm{Int}_{\mathfrak{F}}(G))\simeq 1$ holds for every group $G$ if and only if $\mathfrak{F}$ contains every group $G$ all of whose $\mathfrak{F}$-subgroups are $\mathfrak{F}$-subnormal in $G$. \end{theorem}

Recall that a group is called nilpotent if it coincides with its hypercenter. According to the results of D.~W.~Barnes and O.~H.~Kegel \cite[Chapter IV, Proposition 1.5]{Doerk1992} if $\mathfrak{F}$ is a formation and $G\in \mathfrak{F}$, then $G$ coincides with its $\mathfrak{F}$-hypercenter. The class of all abelian groups shows that in the general case the converse of the previous statment fails. L.~A.~Shemetkov \cite[Question 4.1]{Shemetkov1997} in 1997 posed the problem to describe a family of formations $\mathfrak{F}=(G\mid \mathrm{Z}_\mathfrak{F}(G)=G)$. Such formations are called $Z$-saturated \cite{Murashka2022a}.

\begin{theorem}\label{theorem3} Let $\mathfrak{F}$ be a hereditary formation of soluble groups. Then $\mathfrak{F}$ contains every group $G$ all of whose $\mathfrak{F}$-subgroups are $\mathfrak{F}$-subnormal in $G$ if and only if $\mathfrak{F}$ is a $Z$-saturated formation.\end{theorem}

\begin{corollary}\label{coroll1} Let $\mathfrak{F}$ be a non-empty hereditary formation of soluble groups. Then $\mathrm{Int}_{\mathfrak{F}}(G/\mathrm{Int}_{\mathfrak{F}}(G))\simeq 1$ holds for every group $G$ if and only if $\mathfrak{F}$ is $Z$-saturated. \end{corollary}

\begin{example}
Recall that the rank of a chief factor $H/K=H_1/K\times \ldots \times H_k/K$ in a group $G$ is the number $k$ of isomorphic simple groups $H_i/K$. The class of all soluble groups whose chief factor ranks do not exceed $n$ is a hereditary non-saturated $Z$-saturated formation for $n\geq2$ \cite[Theorem 6]{Murashka2022a}.
\end{example} 

Classes of groups defined by systems of $\mathfrak{F}$-subnormal subgroups are widely studied nowadays (for example, see \cite{Murashka2020, M2014vF, VasAFTI2011, VasVegera2016, Semenchuk2011, MonakhovSokhor2018, GuoSkiba2019Chi, Kamornikov2022}). In \cite{M2014vF} and \cite{VasVegera2016} the classes $v_{\pi}\mathfrak{F}$ and $w_{\pi}\mathfrak{F}$ of groups, whose all cyclic primary and primary $\pi$-subgroups are $\mathfrak{F}$-subnormal were studied respectively. The ideas of these works were combined in \cite{Murashka2020}, where for a saturated homomorph $\mathfrak{H}$ and a hereditary formation $\mathfrak{F}$ the class $f_{\mathfrak{H}}(\mathfrak{F})$ of groups, all of whose $\mathfrak{H}$-subgroups are $\mathfrak{F}$-subnormal, was studied. 

\begin{theorem}\label{theorem4}
Let $\mathfrak{H}$ be a saturated homomorph, $\mathfrak{F}=f_{\mathfrak{H}}(\mathfrak{F})$ be a hereditary formation, $\mathfrak{H}\subseteq \mathfrak{F}$ and $G$ be a group. Then:
  
$(1)$ $\mathrm{Int}_{\mathfrak{F}}(G/N)=\mathrm{Int}_{\mathfrak{F}}(G)/N$ for any $N\unlhd G$ with $N\leq \mathrm{Int}_{\mathfrak{F}}(G)$.
  
$(2)$ $\mathrm{Int}_{\mathfrak{F}}(H)N/N\leq \mathrm{Int}_{\mathfrak{F}}(HN/N)$ for any $H\leq G$ and $N\unlhd G$.

$(3)$ If $L,S\leq G$, then $\mathrm{Int}_{\mathfrak{F}}(L)\cap S\leq \mathrm{Int}_{\mathfrak{F}}(L\cap S)$.
\end{theorem}


In \cite{Murashka2020} it was shown that the constructions $v_{\pi}\mathfrak{F}$ and $w_{\pi}\mathfrak{F}$ do not necessarily lead to saturated formations, i.e.  hereditary $\check{S}$-formations have this property. For the applications of such formations, see \cite[$\S 6.4$]{BallesterBolinches2006}, \cite[$\S 2.4$]{Kamornikov2003} and \cite[$\S 24$]{ShemetkovSkiba1989}.
 
\begin{corollary}\label{coroll4} Let $\mathfrak{F}$ be a non-empty hereditary formation with one of the following conditions: 

$(a)$ $\mathfrak{F}$ is a saturated formation \emph{\cite[Theorem C]{Skiba2012}}; 

$(b)$ $\mathfrak{F}=v_{\pi}\mathfrak{F}$ or $\mathfrak{F}=w_{\pi}\mathfrak{F}$; 

$(c)$ $\mathfrak{F}$ is a $\check{S}$-formation.\\ Then for every group $G$ the following statements hold:

$(1)$ $\mathrm{Int}_{\mathfrak{F}}(G/N)=\mathrm{Int}_{\mathfrak{F}}(G)/N$ for any $N\unlhd G$ with $N\leq \mathrm{Int}_{\mathfrak{F}}(G)$.
  
$(2)$ $\mathrm{Int}_{\mathfrak{F}}(H)N/N\leq \mathrm{Int}_{\mathfrak{F}}(HN/N)$ for any $H\leq G$ and $N\unlhd G$.

$(3)$ If $L,S\leq G$, then $\mathrm{Int}_{\mathfrak{F}}(L)\cap S\leq \mathrm{Int}_{\mathfrak{F}}(L\cap S)$.
\end{corollary} 

Recall that the $\mathfrak{F}$-hypercenter is defined via the action of a group on its chief factors. The similar result for $\mathrm{Int}_{\mathfrak{F}}(G)$ was obtained in \cite{Beidleman2010} when $\mathfrak{F}$ is a hereditary saturated formation and $G$ is a soluble group. 

\begin{theorem}\label{theorem5}
Let $\mathfrak{H}$ be a saturated homomorph, $\mathfrak{F}=f_{\mathfrak{H}}(\mathfrak{F})$ be a hereditary formation, $\mathfrak{H}\subseteq \mathfrak{F}$ and $H/K$ be a chief factor of a group $G$. Then $H/K\leq \mathrm{Int}_{\mathfrak{F}}(G/K)$ if and only if $H/K\rtimes TH/C_G(H/K)\in \mathfrak{F}$ for every $\mathfrak{H}$-subgroup $T/C_G(H/K)$ of $G/C_G(H/K)$.
\end{theorem}

\begin{corollary}\label{coroll8} Let $\mathfrak{F}$ be a non-empty hereditary saturated formation, $G$ be a group and $f$ be its local definition. Then $\mathrm{Int}_{\mathfrak{F}}(G)$ is the largest normal subgroup of $G$ such that $HT/C_G(H/K)\in \mathfrak{N}_pf(p)$ for every $p\in \pi(H/K)$, chief factor $H/K$ of $G$ below it and $\mathfrak{F}$-subgroup $T/C_G(H/K)$ of $G/C_G(H/K)$.     
\end{corollary}

Since $H/K\leq C_G(H/K)$ for every chief factor $H/K$ of a soluble group $G$, from Corollary \ref{coroll8} we get:

\begin{corollary}[{\cite[Main Theorem]{Beidleman2010}}]\label{coroll7} Let $\mathfrak{F}$ be a non-empty hereditary saturated formation, $f$ be a local definition of $\mathfrak{F}$, $G$ be a soluble group and $$g(p)=(L\mid \mathrm{O}_p(L)=1 \, \text{and every} \, \mathfrak{F}\text{-subgroup of} \, L \, \text{belongs to} \, \mathfrak{N}_pf(p)).$$ Then $\mathrm{Int}_{\mathfrak{F}}(G)$ is the largest among normal subgroups $M$ of $G$ with the following property: if $H/K$ is a $p$-chief factor of $G$ with $H\subseteq M$, then $G/C_G(H/K)\in g(p)$. 
\end{corollary} 

\section{Preliminaries} 

The notation and terminology agree with \cite{BallesterBolinches2006, Doerk1992}. Let us recall some of them: $Z_n$ is the cyclic group of order $n$; $A_n$ is the alternating group of degree $n$; $Core_G(H)$ is the largest normal subgroup of the group $G$ that is contained in $H$; if $M$ is a proper subgroup with the property that $M=H$ whenever $M\leq H< G$, then $M$ is called a maximal subgroup of $G$;  a chain of subgroups $H=H_0\subset H_1\subset \ldots \subset H_n=G$ is called maximal in $G$ if either $n=0$ or $H_{i-1}$ is a maximal subgroup of $H_i$ for any $i>0$; a group $G\in \mathfrak{F}$ is called an $\mathfrak{F}$-group; a subgroup $U$ of $G$ is called $\mathfrak{F}$-maximal in $G$ provided $U\in \mathfrak{F}$ and if $U\leq V\leq G$, $V\in \mathfrak{F}$, then $U=V$; $\Phi(G)$ is the Frattini subgroup of $G$, i.e. the intersection of all maximal subgroups of $G$; $\mathrm{F}(G)$ is the Fitting subgroup of $G$, i.e. the largest normal nilpotent subgroup of $G$; $N\rtimes M$ is the semidirect product of $N$ with $M$; $\mathrm{Aut}(G)$ is the group of all automorphisms of a group $G$; $G^\mathfrak{F}$ is the $\mathfrak{F}$-residual of $G$, i.e. the intersection of all those normal subgroups $N$ of $G$ for which $G/N\in \mathfrak{F}$; $G_\mathfrak{S}$ is the soluble radical of $G$, i.e. the largest normal soluble subgroup of $G$.

Recall \cite[Chapter II, Definition 1.1]{Doerk1992} that a class of groups is a collection $\mathfrak{X}$ of groups with the property that if $G\in \mathfrak{X}$ and if $H\simeq G$, then $H\in \mathfrak{X}$. The classes of all soluble groups and of all $p$-groups for a prime $p$ are denoted by $\mathfrak{S}$ and $\mathfrak{N}_p$, respectively. A class of groups $\mathfrak{H}$ is called a homomorph if $G\in \mathfrak{H}$ and $N\unlhd G$, then $G/N\in \mathfrak{H}$. A formation is a homomorph $\mathfrak{F}$ that is closed under taking subdirect products, i.e. from $H/A\in \mathfrak{F}$, $H/B\in \mathfrak{F}$ it follows that $H/(A\cap B)\in \mathfrak{F}$. A class of groups $\mathfrak{F}$ is called: hereditary if $G\in \mathfrak{F}$ and $H\leq G$, then $H\in \mathfrak{F}$; saturated if $G/N\in \mathfrak{F}$ and $N\leq \Phi(G)$, then $G\in \mathfrak{F}$; solubly saturated if $G/N\in \mathfrak{F}$ and $N\leq \Phi(G_\mathfrak{S})$, then $G\in \mathfrak{F}$. The formation product of $\mathfrak{F}_1$ with $\mathfrak{F}_2$ is defined as $\mathfrak{F}_1\mathfrak{F}_2=(G\mid G^{\mathfrak{F}_2}\in \mathfrak{F}_1)$.

Recall that a minimal non-$\mathfrak{F}$-group for a class of groups $\mathfrak{F}$ is a group $G\not \in\mathfrak{F}$, all of whose proper subgroups belong to $\mathfrak{F}$; a Schmidt group is a minimal non-nilpotent group; a formation $\mathfrak{F}$ is called a $\check{S}$-formation if every minimal non-$\mathfrak{F}$-group is either a Schmidt group or a group of prime order (see \cite[$\S 6.4$]{BallesterBolinches2006}).

\begin{theorem}[(D. W. Barnes and O. H. Kegel {\cite[Chapter IV, Proposition 1.5]{Doerk1992}})] \label{BK} Let $\mathfrak{F}$ be a formation, $R/S$ be a normal section of a group $G\in \mathfrak{F}$ and $K$ be a normal subgroup of $G$ contained in $C_G(R/S)$. With respect to the following action of $G/K$ on $R/S$: $(rS)^{gK}=g^{-1}rgS$, $r\in R, g\in G$, form the semidirect product $H=R/S\rtimes G/K$. Then $H\in \mathfrak{F}$. 
\end{theorem}

\begin{lemma}[{\cite[Lemma 6.1.6]{BallesterBolinches2006}}]\label{lem1}
  Let $\mathfrak{F}$ be a formation, $G$ be a group, $H,R\leq G$ and $N\unlhd G$. Then:
  
  $(1)$ If $H/N$ $\mathfrak{F}$-{\rm sn} $G/N$, then $H$ $\mathfrak{F}$-{\rm sn} $G$.
  
  $(2)$ If $H$ $\mathfrak{F}$-{\rm sn} $G$, then $HN/N$ $\mathfrak{F}$-{\rm sn} $G/N$.
  
  $(3)$ If $H$ $\mathfrak{F}$-{\rm sn} $R$ and $R$ $\mathfrak{F}$-{\rm sn} $G$, then $H$ $\mathfrak{F}$-{\rm sn} $G$.
\end{lemma}

The next result directly follows from Lemma~\ref{lem1}:

\begin{lemma}\label{lem123}
  Let $\mathfrak{F}$ be a formation, $G$ be a group, $H,R\leq G$ and $N\unlhd G$. If $H$ $\mathfrak{F}$-{\rm sn} $R$, then $HN$ $\mathfrak{F}$-{\rm sn} $RN$.
\end{lemma}
  
Recall \cite{Kamornikov2003, BallesterBolinches2006} that a subgroup functor is a function $f$ which assigns to each group $G$ a possibly empty set $f(G)$ of subgroups of $G$ satisfying $\theta(f(G))=f(\theta(G))$ for any isomorphism $\theta: G\rightarrow G^*$. The function $s_{n\mathfrak{F}}$ that assigns to each group the set of its $\mathfrak{F}$-subnormal subgroups is a subgroup functor \cite{Kamornikov2003, BallesterBolinches2006}. The next result directly follows from this fact.
  
\begin{lemma}\label{lem111} Let $G$ be a group and $H,R\leq G$. If $\mathfrak{F}$ is non-empty formation and $H$ $\mathfrak{F}$-{\rm sn} $R$, then $H^{\alpha}$ $\mathfrak{F}$-{\rm sn} $R^{\alpha}$ for all $\alpha\in \mathrm{Aut}(G)$.
\end{lemma}

\begin{lemma}[{\cite[Lemma 6.1.7]{BallesterBolinches2006}}]\label{lem11}
  Let $\mathfrak{F}$ be a hereditary formation, $G$ be a group and $H,R\leq G$. If $H$ $\mathfrak{F}$-{\rm sn} $G$, then $H\cap R$ $\mathfrak{F}$-{\rm sn} $R$.
\end{lemma}

A function $f$, which assigns to every prime a formation, is called a formation function. Recall \cite[Chapter IV, Definition 3.1]{Doerk1992} that a formation $\mathfrak{F}$ is called local~if $$\mathfrak{F}=(G \mid G/C_G(H/K) \in f(p) \, \text{for every} \, p\in \pi(H/K)\, \text{and chief factor} \, H/K \, \text{of} \, G)$$ for some formation function $f$. In this case $f$ is called a local definition of $\mathfrak{F}$. According to Gash$\mathrm{\ddot{u}}$ts -- Lubeseder -- Schmid theorem \cite[Chapter IV, Theorem~4.6]{Doerk1992} a non-empty formation is saturated iff it is local. 

Recall \cite[Theorem 1]{Murashka2022a} that $Z\mathfrak{F}=(G \mid G=\mathrm{Z}_{\mathfrak{F}}(G))$ is a $Z$-saturated formation for a class $\mathfrak{F}$ of groups. From this fact and from \cite[Lemma 2.4]{Aivazidis2021} the following result is immediately obtained:

\begin{lemma}\label{coroll2}
Let $\mathfrak{F}$ be a hereditary formation. Then $Z\mathfrak{F}$ is  a hereditary $Z$-saturated formation.
\end{lemma}

\section{The intersection of weak $\mathfrak{F}$-subnormalizers} 

R.~W.~Carter \cite{Carter1962} and C.~J.~Graddon \cite{Graddon1971} studied subnormalizers and $\mathfrak{F}$-subnormalizers respectively.   A subnormalizer or an $\mathfrak{F}$-subnormalizer may not exist in a group for a given subgroup. A.~Mann \cite{Mann1970} proposed the concept of a weak subnormalizer, which always exists, but may not be unique. In \cite{Murashka2022} the concept of a weak K-$\mathfrak{F}$-subnormalizer was proposed. By analogy we propose the following definition:

\begin{definition} We shall call a subgroup $T$ of $G$ a weak $\mathfrak{F}$-subnormalizer of $H$ in $G$ if $H$ $\mathfrak{F}$-sn $T$ and if $H$ $\mathfrak{F}$-sn $M\leq G$ and $T\leq M$, then $T=M$. \end{definition}

Let $\mathfrak{F}$ be a hereditary formation and $H\leq G$. From the definitions of $\mathfrak{F}$-subnormal subgroup and weak $\mathfrak{F}$-subnormalizer it follows that $H$ $\mathfrak{F}$-sn $G$ if and only if $G$ coincides with the weak $\mathfrak{F}$-subnormalizer of $H$ in $G$.

\begin{lemma}\label{lem2}
Let $\mathfrak{F}$ be a non-empty formation, $G$ be a group and $H\leq G$. If $T$ is a weak $\mathfrak{F}$-subnormalizer of $H$ in $G$, then $T^\alpha$ is a weak $\mathfrak{F}$-subnormalizer of $H^\alpha$ in $G$ for any $\alpha\in \mathrm{Aut}(G)$.
\end{lemma}

\begin{proof} Assume that $T$ is a weak $\mathfrak{F}$-subnormalizer of $H$ and $T^\alpha$ is not a weak $\mathfrak{F}$-subnormalizer of $H^\alpha$ for some $\alpha\in \mathrm{Aut}(G)$. Since $H^\alpha$ $\mathfrak{F}$-sn $T^\alpha$ by Lemma~\ref{lem111}, there exists $M\leq G$ such that $H^\alpha$ $\mathfrak{F}$-sn $M$ and $T^\alpha<M$. Then $H$ $\mathfrak{F}$-sn $M^{\alpha^{-1}}$ by Lemma~\ref{lem111} and $T<M^{\alpha^{-1}}$, a contradiction. Therefore $T^\alpha$ is a weak $\mathfrak{F}$-subnormalizer of $H^\alpha$ for any $\alpha\in \mathrm{Aut}(G)$. \end{proof}

Our main idea is to study $\mathrm{Int}_{\mathfrak{F}}(G)$ using the following subgroup: 

\begin{definition} Let $\mathfrak{F}$ be a formation, $\mathfrak{H}$ be a homomorph and $G$ be a group. By $\mathcal{I}_{\mathfrak{H}}^{\mathfrak{F}}(G)$ we shall denote the intersection of all weak $\mathfrak{F}$-subnormalizers of all $\mathfrak{H}$-subgroups of $G$. \end{definition}

\begin{lemma}\label{l7} Let $\mathfrak{H}$ be a saturated homomorph. If $A/N\in \mathfrak{H}$ and $B$ is the minimal supplement to $N$ in $A$, then $B\in \mathfrak{H}$ and $BN/N=A/N$.\end{lemma}

\begin{proof}
  Note that   $B\cap N\subseteq \Phi(B)$  by \cite[Chapter A, Theorem 9.2]{Doerk1992}. Since $A/N=BN/N\simeq B/(B\cap N)\in \mathfrak{H}$, we obtain that $B\in \mathfrak{H}$.
\end{proof}

\begin{theorem}\label{thm1}
Let $\mathfrak{F}$ be a hereditary formation, $\mathfrak{H}$ be a homomorph, $G$ be a group, $N\unlhd G$ and $L\leq G$. The following statements hold:

$(1)$ $\mathcal{I}_{\mathfrak{H}}^{\mathfrak{F}}(G)$ is the characteristic subgroup of $G$ and is the largest  subgroup among normal subgroups $N_1$ of $G$ with $H$ $\mathfrak{F}$-{\rm sn} $HN_1$ for any $\mathfrak{H}$-subgroup $H$ of $G$.

$(2)$ If $\mathfrak{H}$ is saturated, then $\mathcal{I}_{\mathfrak{H}}^{\mathfrak{F}}(L)N/N\leq \mathcal{I}_{\mathfrak{H}}^{\mathfrak{F}}(LN/N)$.

$(3)$ If $N\leq \mathcal{I}_{\mathfrak{H}}^{\mathfrak{F}}(G)$, then $\mathcal{I}_{\mathfrak{H}}^{\mathfrak{F}}(G/N)\leq \mathcal{I}_{\mathfrak{H}}^{\mathfrak{F}}(G)/N$, i.e. $\mathcal{I}_{\mathfrak{H}}^{\mathfrak{F}}(G/\mathcal{I}_{\mathfrak{H}}^{\mathfrak{F}}(G)) \simeq 1$. If $\mathfrak{H}$ is saturated, then $\mathcal{I}_{\mathfrak{H}}^{\mathfrak{F}}(G/N)= \mathcal{I}_{\mathfrak{H}}^{\mathfrak{F}}(G)/N$.

$(4)$ If $S\leq G$, then $\mathcal{I}_{\mathfrak{H}}^{\mathfrak{F}}(L)\cap S\leq \mathcal{I}_{\mathfrak{H}}^{\mathfrak{F}}(L\cap S)$.

$(5)$ $\mathcal{I}_{\mathfrak{H}}^{\mathfrak{F}}(\mathcal{I}_{\mathfrak{H}}^{\mathfrak{F}}(G))=\mathcal{I}_{\mathfrak{H}}^{\mathfrak{F}}(G)$.
\end{theorem}

\begin{proof} $(1)$ Let $\mathcal{K}$ be the set of all weak $\mathfrak{F}$-subnormalizers of $\mathfrak{H}$-subgroups of $G$ and $T\in \mathcal{K}$. Then $T$ is a weak $\mathfrak{F}$-subnormalizer of some $\mathfrak{H}$-subgroup $H$. Since $\mathfrak{H}$ is a class of group, $H^{\alpha}\in \mathfrak{H}$ for any $\alpha\in \mathrm{Aut}(G)$. Then $T^\alpha\in \mathcal{K}$ for any $\alpha\in \mathrm{Aut}(G)$ by Lemma~\ref{lem2}. Note that $$(\mathcal{I}_{\mathfrak{H}}^{\mathfrak{F}}(G))^{\alpha}=(\cap_{T\in \mathcal{K}}T)^{\alpha}=(\cap_{T\in \mathcal{K}}T^{\alpha^{-1}})^\alpha=\cap_{T\in \mathcal{K}}T^{\alpha^{-1}\alpha}=\cap_{T\in \mathcal{K}}T=\mathcal{I}_{\mathfrak{H}}^{\mathfrak{F}}(G)$$ for any $\alpha\in \mathrm{Aut}(G)$. Therefore $\mathcal{I}_{\mathfrak{H}}^{\mathfrak{F}}(G)$ is the characteristic subgroup of $G$.

Let $H$ be an $\mathfrak{H}$-subgroup of $G$ and $N_1$ be a normal subgroup of $G$ such that $H_1$ $\mathfrak{F}$-sn $H_1N_1$ for any $\mathfrak{H}$-subgroup $H_1$ of $G$. If $T$ is a weak $\mathfrak{F}$-subnormalizer of $H$ in $G$, then $HN_1$ $\mathfrak{F}$-sn $TN_1$ by Lemma~\ref{lem123}. Hence $H$ $\mathfrak{F}$-sn $TN_1$ by $(3)$ of Lemma~\ref{lem1}. So $TN_1=T$ by the definition of a weak $\mathfrak{F}$-subnormalizer. Thus $N_1\leq \mathcal{I}_{\mathfrak{H}}^{\mathfrak{F}}(G)$.

On the other hand, since $\mathfrak{F}$ is a hereditary formation and $H\mathcal{I}_{\mathfrak{H}}^{\mathfrak{F}}(G)$ is contained in every weak $\mathfrak{F}$-subnormalizer of $H$ in $G$, it follows that $H$ $\mathfrak{F}$-sn $H\mathcal{I}_{\mathfrak{H}}^{\mathfrak{F}}(G)$ by Lemma~\ref{lem11}. Then $\mathcal{I}_{\mathfrak{H}}^{\mathfrak{F}}(G)$ is the largest  subgroup among normal subgroups $N_1$ of $G$ with $H_1$ $\mathfrak{F}$-sn $H_1N_1$ for any $\mathfrak{H}$-subgroup $H_1$ of $G$. 

$(2)$ Let $L=G$ and $I=\mathcal{I}_{\mathfrak{H}}^{\mathfrak{F}}(G)$. For any $V/N\in\mathfrak{H}$ there exists  $L_1\in\mathfrak{H}$ with $L_1N=V$ by Lemma~\ref{l7}. So $L_1$ $\mathfrak{F}$-sn $L_1I$ by (1). Therefore $L_1N/N=V/N$ $\mathfrak{F}$-sn $VI/N$ by Lemma~\ref{lem1}. Hence $V/N$ $\mathfrak{F}$-sn $VI/N$ for any $V/N\in \mathfrak{H}$. Thus $\mathcal{I}_{\mathfrak{H}}^{\mathfrak{F}}(G)N/N\leq \mathcal{I}_{\mathfrak{H}}^{\mathfrak{F}}(G/N)$ by $(1)$.      
  
Now let $L$ be any subgroup of $G$ and let $f:L/(L\cap N)\rightarrow LN/N$ be the canonical isomorphism. Then $$f(\mathcal{I}_{\mathfrak{H}}^{\mathfrak{F}}(L/(L\cap N)))=\mathcal{I}_{\mathfrak{H}}^{\mathfrak{F}}(LN/N)\, \text{and}\, f(\mathcal{I}_{\mathfrak{H}}^{\mathfrak{F}}(L)(L\cap N)/(L\cap N))=\mathcal{I}_{\mathfrak{H}}^{\mathfrak{F}}(L)N/N$$ by $(1)$. From our previous results we get $$\mathcal{I}_{\mathfrak{H}}^{\mathfrak{F}}(L)(L\cap N)/(L\cap N)\leq \mathcal{I}_{\mathfrak{H}}^{\mathfrak{F}}(L/(L\cap N)).$$ Hence $\mathcal{I}_{\mathfrak{H}}^{\mathfrak{F}}(L)N/N\leq \mathcal{I}_{\mathfrak{H}}^{\mathfrak{F}}(LN/N)$. 

$(3)$ Let $N\leq \mathcal{I}_{\mathfrak{H}}^{\mathfrak{F}}(G)$ and $\overline{K}=\mathcal{I}_{\mathfrak{H}}^{\mathfrak{F}}(G/N)=K/N$. Note that $K\unlhd G$. Let $H\leq G$ and $H\in \mathfrak{H}$. Hence $\overline{H}=HN/N\simeq H/(H\cap N)\in \mathfrak{H}$. Let $\overline{T}$ be a weak $\mathfrak{F}$-subnormalizer of $\overline{H}$ in $G/N$, i.e. $\overline{H}$ $\mathfrak{F}$-sn $\overline{T}$. From $\overline{H}< \overline{KH}\leq \overline{T}$ it follows that $\overline{H}=HN/N$ $\mathfrak{F}$-sn $\overline{KN}=HK/N$. So $HN$ $\mathfrak{F}$-sn $HK$ by $(2)$ of Lemma \ref{lem1}. Since $H$ $\mathfrak{F}$-sn $HN$ by Lemma \ref{lem11}, we obtain that $H$ $\mathfrak{F}$-sn $HK$ by $(3)$ of Lemma~\ref{lem1}. Let $T_1$ be a weak $\mathfrak{F}$-subnormalizer of $H$ in $G$. Then $HK$ $\mathfrak{F}$-sn $T_1K$ by Lemma~\ref{lem123}. Since $H$ $\mathfrak{F}$-sn $HK$ and $HK$ $\mathfrak{F}$-sn $T_1K$,  $H$ $\mathfrak{F}$-sn $T_1K$ by $(3)$ of Lemma~\ref{lem1}. So $T_1K=T_1$ by the definition of weak $\mathfrak{F}$-subnormalizer. Therefore $K$ lies in all weak $\mathfrak{F}$-subnormalizers of all $\mathfrak{H}$-subgroups of $G$. Hence $K\leq I$, i.e. $\mathcal{I}_{\mathfrak{H}}^{\mathfrak{F}}(G/N)\leq \mathcal{I}_{\mathfrak{H}}^{\mathfrak{F}}(G)/N$. 

If $N=\mathcal{I}_{\mathfrak{H}}^{\mathfrak{F}}(G)$, then $N/N\leq \mathcal{I}_{\mathfrak{H}}^{\mathfrak{F}}(G/N)\leq N/N$. So $\mathcal{I}_{\mathfrak{H}}^{\mathfrak{F}}(G/\mathcal{I}_{\mathfrak{H}}^{\mathfrak{F}}(G)) \simeq 1$. 

If $\mathfrak{H}$ is a saturated homomorph, then $\mathcal{I}_{\mathfrak{H}}^{\mathfrak{F}}(G)/N\leq \mathcal{I}_{\mathfrak{H}}^{\mathfrak{F}}(G/N)$ by $(2)$. Thus $\mathcal{I}_{\mathfrak{H}}^{\mathfrak{F}}(G)/N= \mathcal{I}_{\mathfrak{H}}^{\mathfrak{F}}(G/N)$. 

$(4)$ Note that $K$ $\mathfrak{F}$-sn $K\mathcal{I}_{\mathfrak{H}}^{\mathfrak{F}}(L)$ for any $\mathfrak{H}$-subgroup $K$ of $L\cap S$ by $(1)$ and $K(\mathcal{I}_{\mathfrak{H}}^{\mathfrak{F}}(L)\cap S)\leq K\mathcal{I}_{\mathfrak{H}}^{\mathfrak{F}}(L)$. So $K$ $\mathfrak{F}$-sn $K(\mathcal{I}_{\mathfrak{H}}^{\mathfrak{F}}(L)\cap S)$ for any $\mathfrak{H}$-subgroup $K$ of $L\cap S$. Thus $\mathcal{I}_{\mathfrak{H}}^{\mathfrak{F}}(L)\cap S\leq \mathcal{I}_{\mathfrak{H}}^{\mathfrak{F}}(L\cap S)$ by $(1)$.

$(5)$ From $(4)$ it follows that $$\mathcal{I}_{\mathfrak{H}}^{\mathfrak{F}}(\mathcal{I}_{\mathfrak{H}}^{\mathfrak{F}}(G))\leq \mathcal{I}_{\mathfrak{H}}^{\mathfrak{F}}(G)=\mathcal{I}_{\mathfrak{H}}^{\mathfrak{F}}(G)\cap \mathcal{I}_{\mathfrak{H}}^{\mathfrak{F}}(G)\leq \mathcal{I}_{\mathfrak{H}}^{\mathfrak{F}}(G\cap \mathcal{I}_{\mathfrak{H}}^{\mathfrak{F}}(G))=\mathcal{I}_{\mathfrak{H}}^{\mathfrak{F}}(\mathcal{I}_{\mathfrak{H}}^{\mathfrak{F}}(G)).$$ Thus $\mathcal{I}_{\mathfrak{H}}^{\mathfrak{F}}(\mathcal{I}_{\mathfrak{H}}^{\mathfrak{F}}(G)$)$=\mathcal{I}_{\mathfrak{H}}^{\mathfrak{F}}(G)$. \end{proof}

The key idea in the proof of the main results of the work is:

\begin{theorem}\label{thm2} Let $\mathfrak{H}$ be a homomorph, $\mathfrak{F}$ be a hereditary formation, $\mathfrak{H}\subseteq \mathfrak{F}$ and $G$ be a group. Then $\mathcal{I}_{\mathfrak{H}}^{\mathfrak{F}}(G)=\mathrm{Int}_{\mathfrak{F}}(G)$ holds for every group $G$ if and only if $\mathfrak{F}=f_{\mathfrak{H}}(\mathfrak{F})$. \end{theorem}

\begin{proof}
Assume that $\mathfrak{F}=f_{\mathfrak{H}}(\mathfrak{F})$. Then $1\in \mathfrak{F}$. So $\mathfrak{F}\neq \emptyset$. Let $M$ be an $\mathfrak{F}$-maximal subgroup of $G$, $I=\mathcal{I}_{\mathfrak{H}}^{\mathfrak{F}}(G)$ and $H$ be an $\mathfrak{H}$-subgroup of $MI$. Then $H$ $\mathfrak{F}$-sn $HI$ by $(1)$ of Theorem~\ref{thm1}. Note that $MI/I\simeq M/(M\cap I) \in \mathfrak{F}$. Since $\mathfrak{F}$ is a hereditary formation, $HI/I$ $\mathfrak{F}$-sn $MI/I$. Hence $HI$ $\mathfrak{F}$-sn $MI$ by $(1)$ of Lemma~\ref{lem1}. Thus $H$ $\mathfrak{F}$-sn $MI$ for any $\mathfrak{H}$-subgroup $H$ of $MI$. Therefore $MI\in \mathfrak{F}$ by our assumption. Hence $MI=M$. So $I\subseteq \mathrm{Int}_{\mathfrak{F}}(G)$.

Let $H$ be an $\mathfrak{H}$-subgroup of $G$. Then from $\mathfrak{H}\subseteq \mathfrak{F}$ it follows that $H$ is contained in some $\mathfrak{F}$-maximal subgroup of $G$. So $H\mathrm{Int}_{\mathfrak{F}}(G) \in \mathfrak{F}$. Hence $H$ $\mathfrak{F}$-sn $H\mathrm{Int}_{\mathfrak{F}}(G)$ for any $\mathfrak{H}$-subgroup $H$ of $G$. Now $\mathrm{Int}_{\mathfrak{F}}(G)\subseteq I$ by $(1)$ of Theorem~\ref{thm1}. Thus $I=\mathrm{Int}_{\mathfrak{F}}(G)$.

Assume that $\mathcal{I}_{\mathfrak{H}}^{\mathfrak{F}}(G)=\mathrm{Int}_{\mathfrak{F}}(G)$ holds for every group $G$. Since $\mathfrak{F}$ is a hereditary formation, it is obvious that $\mathfrak{F}\subseteq f_{\mathfrak{H}}(\mathfrak{F})$. Let $G\in f_{\mathfrak{H}}(\mathfrak{F})$. Then $G=\mathcal{I}_{\mathfrak{H}}^{\mathfrak{F}}(G)$. Hence $G=\mathrm{Int}_{\mathfrak{F}}(G)$ by our assumption. It means that $G\in \mathfrak{F}$. Therefore $f_\mathfrak{H}(\mathfrak{F})\subseteq \mathfrak{F}$. Thus $\mathfrak{F}=f_{\mathfrak{H}}(\mathfrak{F})$.
\end{proof}

\section{Prooves of the main results}

\begin{proof}[Proof of the Theorem~\ref{theorem1}] Let $\mathfrak{F}_1$ be a class of groups, all of whose composition factors are from $(Z_2, Z_3, Z_5, A_5)$. It is clear that this class is a saturated formation and therefore it is solubly saturated.
Recall \cite[Chapter~1, Definition 3.1]{Guo2015book} that a group $G$ is called a quasisoluble if for every $\mathfrak{S}$-eccentric (i.e. non $\mathfrak{S}$-central) chief factor $H/K$ of $G$, every automorphism of $H/K$ induced by an element of $G$ is inner. Let $\mathfrak{F}_2$ be the class of all quasisoluble groups. Then $\mathfrak{F}_2$ is a solubly saturated formation by \cite[Chapter~1,~Corollary~3.8]{Guo2015book}.

Let $\mathfrak{F}=\mathfrak{F}_1\cap \mathfrak{F}_2$. Then it is solubly saturated as the intersection of solubly saturated formations. So    $G/\mathrm{Z}_{\mathfrak{S}}(G)$ is a direct product of groups isomorphic to $A_5$ for any $\mathfrak{F}$-group $G$ by \cite[Chapter 1, Theorem 3.12(2)]{Guo2015book}. Let $H$ be a subgroup of an $\mathfrak{F}$-group $G$. Then $H\mathrm{Z}_{\mathfrak{S}}(G)/\mathrm{Z}_{\mathfrak{S}}(G)$ is a subgroup of the direct product of groups isomorphic to $A_5$. Since all proper subgroups of $A_5$ and $A_5$ itself belong to $\mathfrak{F}$, the projection of $H\mathrm{Z}_{\mathfrak{S}}(G)/\mathrm{Z}_{\mathfrak{S}}(G)$ onto each of the factors of the direct product belongs to $\mathfrak{F}$. Note that $\mathfrak{F}$ is closed under subdirect products and the intersection of the kernels of the given projections is isomorphic to 1. Hence $H\mathrm{Z}_{\mathfrak{S}}(G)/\mathrm{Z}_{\mathfrak{S}}(G)\in \mathfrak{F}$. Note that $H\mathrm{Z}_{\mathfrak{S}}(G)/\mathrm{Z}_{\mathfrak{S}}(G)\simeq H/(\mathrm{Z}_{\mathfrak{S}}(G)\cap H)$ and $\mathrm{Z}_{\mathfrak{S}}(G)\cap H \subseteq \mathrm{Z}_{\mathfrak{S}}(H)$ by \cite[Chapter~1,~Theorem~2.7(a)]{Guo2015book}. Since the class of soluble $\{2,3,5\}$-groups is contained in $\mathfrak{F}$, the soluble hypercenter of $H$ lies in the $\mathfrak{F}$-hypercenter of $H$. So $H\in \mathfrak{F}$. Thus $\mathfrak{F}$ is a hereditary formation. 

According to the result of R. L. Griess and P. Schmid \cite[Appendix $\beta$, $\beta.9$]{Doerk1992} for $p=2$ and $G\simeq A_5$ there exists a Frattini $\mathbb{F}_2G$-module $A$ which is faithful for $G$. By the well-known theorem of W. Gasch$\mathrm{\ddot{u}}$ts \cite[Appendix $\beta$]{Doerk1992} there exists a Frattini extension $A\rightarrowtail R \twoheadrightarrow G$ such that $A\simeq \Phi(R)$ and $R/\Phi(R)\simeq G$. Note that   all maximal subgroups of $G$ are soluble and the class of soluble groups is closed under extensions. Hence all maximal subgroups of $R$ are soluble. Since $A$ is a faithful module, there exists a chief factor $H/K$ of $R$ such that $R/C_R(H/K)=G$ is not soluble. Now $H/K$ is an abelian noncentral chief factor of $R$. Hence  $R$ does not act by inner automorphisms on it. Therefore $R\not \in \mathfrak{F}$. So $\mathrm{Int}_{\mathfrak{F}}(R)=A$ but $R/A=G\in \mathfrak{F}$. Thus $\mathrm{Int}_{\mathfrak{F}}(R/\mathrm{Int}_{\mathfrak{F}}(R))\not \simeq 1$.
\end{proof}

\begin{proof}[Proof of the Theorem~\ref{theorem2}] Assume that $\mathrm{Int}_{\mathfrak{F}}(G/\mathrm{Int}_{\mathfrak{F}}(G))\simeq 1$ holds for any group $G$ and there exists a non-$\mathfrak{F}$-group all of whose $\mathfrak{F}$-subgroups are $\mathfrak{F}$-subnormal. Let $G$ be a minimal order group with this property. Since $\mathfrak{F}$ is a hereditary formation, all $\mathfrak{F}$-subgroups of every maximal subgroup of $G$ are $\mathfrak{F}$-subnormal in it by Lemma~\ref{lem11}. Hence $M\in \mathfrak{F}$ for every maximal subgroup $M$ of $G$ by our assumption. Therefore $\mathrm{Int}_{\mathfrak{F}}(G)=\Phi(G)$. Assume that $\Phi(G)\neq 1$. If $G/\Phi(G)\in \mathfrak{F}$, then $1\simeq \mathrm{Int}_{\mathfrak{F}}(G/\mathrm{Int}_{\mathfrak{F}}(G))=\mathrm{Int}_{\mathfrak{F}}(G/\Phi(G))=G/\Phi(G)$, a contradiction. Thus $G/\Phi(G)\not\in \mathfrak{F}$. Let $H/\Phi(G)$ be a proper $\mathfrak{F}$-subgroup of $G/\Phi(G)$. Then $H$ is a proper subgroup of $G$, i.e. $H\in \mathfrak{F}$. Thus $H/\Phi(G)$ $\mathfrak{F}$-sn $G/\Phi(G)$ by (2) of Lemma~\ref{lem1}. It means that $G/\Phi(G)\in \mathfrak{F}$ by our assumption, a contradiction. Thus, $\Phi(G)= 1$. Since $1$ $\mathfrak{F}$-sn $G$, we obtain that $G^{\mathfrak{F}}<G$. From $\Phi(G)= 1$ it follows that there is a maximal subgroup $M$ of $G$ with $G^{\mathfrak{F}}\nleq M$. Since $M$ is $\mathfrak{F}$-subnormal and maximal in $G$, it follows that $G^{\mathfrak{F}}\leq M$, the final contradiction.

Assume that $\mathfrak{F}$ contains every group $G$ all of whose $\mathfrak{F}$-subgroups are $\mathfrak{F}$-subnormal in $G$. Then $\mathrm{Int}_{\mathfrak{F}}(G/\mathrm{Int}_{\mathfrak{F}}(G))\simeq 1$ holds for every group $G$ by Theorems~\ref{thm1}~and~\ref{thm2}.
\end{proof}

\begin{lemma}\label{Stat2}
   Let $\mathfrak{H}$ be a  homomorph and $\mathfrak{F}=f_\mathfrak{H}(\mathfrak{F})$ be a hereditary formation. Then  $\mathfrak{F}$  is a hereditary   $Z$-saturated formation.
\end{lemma}

\begin{proof} Note that  $Z\mathfrak{F}$ is a hereditary $Z$-saturated formation by Lemma \ref{coroll2}.
Assume that   $\mathfrak{F}$ is not  $Z$-saturated. Hence $\mathfrak{F}\subset Z\mathfrak{F}$. 
Let   $G$ be a minimal order group from  $Z\mathfrak{F}\setminus\mathfrak{F}$. Then all proper subgroups and quotient groups of $G$ belong to   $\mathfrak{F}$. Since $\mathfrak{F}$ is a formation,   $G$  has the unique minimal normal subgroup $N$. If  $N$ is non-abelian, then  $G/C_G(N)\simeq G$. From $N\rtimes (G/C_G(N))\in\mathfrak{F}$ it follows that $G\in\mathfrak{F}$, a contradiction. So $N$ is abelian. If   $G$ is a primitive group, then $G\simeq N\rtimes G/C_G(N)\in\mathfrak{F}$ by \cite[Proposition 1.1.12]{BallesterBolinches2006}, a contradiction. Hence $G$ is not a primitive group. Therefore $N\leq \Phi(G)$.

If  $G\in\mathfrak{H}$, then $G\,\mathfrak{F}$-${\rm sn}\,G$. Let $H$ be a proper $\mathfrak{H}$-subgroup of  $G$. Then $HN<G$. Hence  $HN\in\mathfrak{F}$. So $H\,\mathfrak{F}$-${\rm sn}\,HN$. From $HN/N\simeq H/(H\cap N)\in\mathfrak{H}$ and $G/N\in\mathfrak{F}=f_{\mathfrak{H}}(\mathfrak{F})$ it follows that    $HN/N\,\mathfrak{F}$-${\rm sn}\,G/N$. So   $HN \,\mathfrak{F}$-${\rm sn}\,G $ by Lemma \ref{lem1}. Now $H \,\mathfrak{F}$-${\rm sn}\,G $ by Lemma \ref{lem1}. Thus $G\in f_{\mathfrak{H}}(\mathfrak{F})=\mathfrak{F}$, the final contradiction.
\end{proof}

\begin{proof}[Proof of the Theorem~\ref{theorem3}] Let $\mathfrak{F}=f_\mathfrak{F}(\mathfrak{F})$. 
Since $\mathfrak{F}$ is a homomorph, $\mathfrak{F}$ is a $Z$-saturated formation by Lemma \ref{Stat2}.

Let $\mathfrak{F}$ be a $Z$-saturated formation. It is clear that $ (1) \subseteq \mathfrak{F}\subseteq f_\mathfrak{F}(\mathfrak{F})$. Assume that $\mathfrak{F}\neq f_{\mathfrak{F}}(\mathfrak{F})$. Let $G$ be a minimal order group from $f_\mathfrak{F}(\mathfrak{F})\setminus \mathfrak{F}$. Then all maximal subgroups of $G$ belong to $\mathfrak{F}$ by Lemma~\ref{lem11} and our assumption. Since 1 $\mathfrak{F}$-sn $G$ and $\mathfrak{F}\subseteq \mathfrak{S}$, we see that $G\in \mathfrak{S}$. Note that if all subgroups of $G$ are $\mathfrak{F}$-subnormal, then all subgroups of $G/N$ are also $\mathfrak{F}$-subnormal by (2) of Lemma~\ref{lem1}. If $|N|>1$, then $G/N\in \mathfrak{F}$ by assumption. So $G$ has a unique minimal normal subgroup $N$ and $G/N\in \mathfrak{F}$. Thus $N=G^{\mathfrak{F}}$. Since $M$ $\mathfrak{F}$-sn $G$, we obtain that $N\leq M$ for every maximal subgroup $M$ of $G$. Thus $N\leq \Phi(G)$. Since $G$ is soluble, $\Phi(G)< \mathrm{F}(G)$ and $\mathrm{F}(G)\leq C_G(N)$ \cite[Chapter A, Theorem 10.6(b)]{Doerk1992}. Let $M$ be a maximal subgroup of $G$ such that $M\mathrm{F}(G)=G$. Then $G=MC_G(N)$. Hence  $N$ is a minimal normal subgroup of $M$. Since $\mathfrak{F}$ is a formation and $M\in \mathfrak{F}$, we have $N\rtimes G/C_G(N)\simeq N\rtimes M/C_G(N)\in \mathfrak{F}$ by Theorem~\ref{BK}. Thus $N$ is an $\mathfrak{F}$-central chief factor of $G$. Since $G/N\in \mathfrak{F}$ and  $\mathfrak{F}$ is $Z$-saturated, we see that $G\in \mathfrak{F}$, the final contradiction.
\end{proof}

\begin{proof}[Proof of the Theorem~\ref{theorem4}] If $\mathfrak{F}=f_{\mathfrak{H}}(\mathfrak{F})$, then $1\in \mathfrak{F}$. Note that $\mathrm{Int}_{\mathfrak{F}}(G)=\mathcal{I}_{\mathfrak{H}}^{\mathfrak{F}}(G)$ holds for every group $G$ by Theorem~\ref{thm2}. Now statements (1), (2), (3) of Theorem~\ref{theorem4} directly follow from statements (2), (3), (4) of Theorem~\ref{thm1}.
\end{proof}

\begin{lemma}\label{122} Let $\mathfrak{F}$ be a non-empty hereditary saturated formation. Then $\mathfrak{F}=f_\mathfrak{F}(\mathfrak{F})$.
\end{lemma}

\begin{proof}
Since $\mathfrak{F}$ is hereditary, it is clear that $\mathfrak{F}\subseteq f_\mathfrak{F}(\mathfrak{F})$. Assume that $f_\mathfrak{F}(\mathfrak{F})\setminus \mathfrak{F}\neq\emptyset$. Let $G$ be a minimal order group from $f_\mathfrak{F}(\mathfrak{F})\setminus \mathfrak{F}$. Hence all proper subgroups and quotient groups of $G$ belong to $\mathfrak{F}$  by Lemmas~\ref{lem1}~and~\ref{lem11}. Let $M$ be a maximal subgroup of $G$. Then $M\in \mathfrak{F}$ and $MG^\mathfrak{F}<G$. Hence $G^\mathfrak{F}\leq M$. Therefore $G^\mathfrak{F}\leq\Phi(G)$. From $G\not\in \mathfrak{F}$ it follows that $\Phi(G)\neq 1$. Now $G/\Phi(G)\in \mathfrak{F}$. Since $\mathfrak{F}$ is saturated, we see that $G\in \mathfrak{F}$, the contradiction. Thus $\mathfrak{F}=f_\mathfrak{F}(\mathfrak{F})$.
\end{proof}

\begin{proof}[Proof of the Corollary~\ref{coroll4}] $(a)$ Let $\mathfrak{H}=\mathfrak{F}\subseteq \mathfrak{F}$. Then $\mathfrak{F}=f_\mathfrak{F}(\mathfrak{F})$ by Lemma~\ref{122}. Therefore (a) of Corollary~\ref{coroll4} follows from Theorem~\ref{theorem4}. 

$(b)$ Let $\mathfrak{F}=v_{\pi}\mathfrak{F}$. Then $\mathfrak{F}=v_{\pi\cap \pi(\mathfrak{F})}\mathfrak{F}$ by \cite[Theorem A(2)]{M2014vF}. Let $\mathfrak{H}=(G\mid G$ is a cyclic primary $(\pi\cap \pi(\mathfrak{F}))$-group$)$. Since $\mathfrak{H}\subseteq \mathfrak{N}_{\pi\cap \pi(\mathfrak{F})}\subseteq \mathfrak{F}$ by \cite[Theorem A(1)]{M2014vF}, we have that the case $\mathfrak{F}=v_{\pi}\mathfrak{F}$ of Corollary~\ref{coroll4} follows from Theorem~\ref{theorem4}. 

Let $\mathfrak{F}=w_{\pi}\mathfrak{F}$. Then $\mathfrak{F}=w_{\pi\cap \pi(\mathfrak{F})}\mathfrak{F}$ by \cite[Theorem 3.1(4)]{VasVegera2016}. Let $\mathfrak{H}=(G\mid G$ is a primary $(\pi\cap \pi(\mathfrak{F}))$-group$)$. Since $\mathfrak{H}\subseteq \mathfrak{N}_{\pi\cap \pi(\mathfrak{F})}\subseteq \mathfrak{F}$ by \cite[Theorem 3.1(3)]{VasVegera2016}, we have that the case $\mathfrak{F}=w_{\pi}\mathfrak{F}$ of Corollary~\ref{coroll4} follows from Theorem~\ref{theorem4}.

$(c)$ Since $\mathfrak{F}=v\mathfrak{F}$ by \cite[Corollary 3.9]{Murashka2020}, (c) follows from (b).
\end{proof}

\begin{proof}[Proof of the Theorem~\ref{theorem5}] If $\mathfrak{F}=f_{\mathfrak{H}}(\mathfrak{F})$, then $1\in \mathfrak{F}$. Assume that $H/K\leq \mathrm{Int}_{\mathfrak{F}}(G/K)$. Note that $C_G(H/K)/K=C_{G/K}(H/K)$. If $$T/C_G(H/K)\simeq (T/K)/(C_G(H/K)/K)=(T/K)/C_{G/K}(H/K)\in \mathfrak{H},$$ then by Lemma~\ref{l7} there exists $L/K\in \mathfrak{H}$ such that $$(L/K)C_{G/K}(H/K)/C_{G/K}(H/K)\simeq (T/K)/C_{G/K}(H/K).$$ From $\mathfrak{H}\subseteq \mathfrak{F}$ it follows that there is an $\mathfrak{F}$-maximal subgroup $M/K$ of $G/K$ with $L/K\leq M/K$. Since $H\leq M$ and $\mathfrak{F}$ is hereditary, we see that $LH/K\in \mathfrak{F}$. Hence  $H/K\rtimes (LH/K)/C_{LH/K}(H/K)\in \mathfrak{F}$ by Theorem~\ref{BK}. Therefore $$H/K\rtimes (LH/K)/C_{LH/K}(H/K)=H/K\rtimes (LH/K)/(C_{G/K}(H/K)\cap (LH/K))\simeq$$ $$H/K\rtimes (LH/K)C_{G/K}(H/K)/C_{G/K}(H/K)\simeq$$ $$ H/K\rtimes (LHC_{G}(H/K)/K)/(C_{G}(H/K)/K)\simeq$$ $$H/K\rtimes LHC_{G}(H/K)/C_{G}(H/K) = H/K\rtimes TH/C_G(H/K)\in \mathfrak{F}.$$

Assume now $H/K\rtimes TH/C_G(H/K)\in \mathfrak{F}$ for any $\mathfrak{H}$-subgroup $T/C_G(H/K)$ of $G/C_G(H/K)$. Let $L/K$ be an $\mathfrak{H}$-subgroup of $G/K$. Then $$LC_G(H/K)/C_G(H/K)\simeq (L/K)C_{G/K}(H/K)/C_{G/K}(H/K)\in \mathfrak{H}.$$ So $H/K\rtimes LHC_G(H/K)/C_G(H/K)\in \mathfrak{F}$. Then $$H/K\rtimes LHC_G(H/K)/C_G(H/K)\simeq H/K\rtimes LH/C_{LH}(H/K)\in \mathfrak{F}.$$

Let $H_1/K_1$ be the chief factor of $LH/K$ below $H/K$. Since $\mathfrak{F}$ is a hereditary formation, we have $H_1/K\rtimes LH/C_{LH}(H/K)\in \mathfrak{F}$. Note that $K_1/K\unlhd H_1/K\rtimes LH/C_{LH}(H/K)$. So $H_1/K_1\rtimes LH/C_{LH}(H/K)\in \mathfrak{F}$. Note that $C_{LH}(H/K)\leq C_{LH}(H_1/K_1)$ and $C_{LH}(H_1/K_1)/C_{LH}(H/K)\unlhd H_1/K_1\rtimes LH/C_{LH}(H/K)$. So $$H_1/K_1\rtimes (LH/C_{LH}(H/K))/(C_{LH}(H_1/K_1)/C_{LH}(HK/K))\simeq$$ $$ H_1/K_1\rtimes LH/C_{LH}(H_1/K_1)\in \mathfrak{F}.$$ Thus all chief factors of $LH/K$ below $H/K$ are $\mathfrak{F}$-central. From $$(LH/K)/(H/K)\simeq LH/H\simeq L/(L\cap H)\in \mathfrak{H}\subseteq \mathfrak{F}$$ it follows that $LH/K\in \mathfrak{F}$ by Lemma \ref{Stat2}. So $L/K$ $\mathfrak{F}$-sn $LH/K$ for any $\mathfrak{H}$-subgroup $L/K$ of $G/K$. Hence $H/K\leq \mathcal{I}_{\mathfrak{H}}^{\mathfrak{F}}(G/K)= \mathrm{Int}_{\mathfrak{F}}(G/K)$ by Theorems~\ref{thm1}~and~\ref{thm2}.
\end{proof}

\begin{proof}[Proof of the Corollary~\ref{coroll8}] Let $\mathfrak{F}$ be a non-empty hereditary saturated formation and $f$ be its local definition (it exists by Gash$\mathrm{\ddot{u}}$ts -- Lubeseder -- Schmid theorem). Now $\mathfrak{F}=f_\mathfrak{F}(\mathfrak{F})$ by Lemma~\ref{122}. From Theorem~\ref{theorem5} it follows that we need only to prove that $H/K\rtimes TH/C_G(H/K)\in \mathfrak{F}$ iff $TH/C_G(H/K)\in \mathfrak{N}_pf(p)$ for every $p\in \pi(H/K)$ where $H/K$ is a chief factor of $G$ and $T/C_G(H/K)$ is an $\mathfrak{F}$-subgroup of $G/C_G(H/K)$. Let $C=C_G(H/K)$ and $\overline{T}=(H/K)\rtimes (HT/C)$.

Assume that $TH/C_G(H/K)\in \mathfrak{N}_pf(p)$ for all $p\in \pi(H/K)$. Suppose that $H/K$ is non-abelian. Since $\mathfrak{N}_pf(p)$ is a formation, $(TH/C)/C_{TH/C}(\overline{U})\in \mathfrak{N}_pf(p)$ for every chief factor $\overline{U}$ of $TH/C$ below $HC/C$. From $$(TH/C)/(HC/C)\simeq TH/HC\simeq T/(T\cap HC)=T/C(T\cap H)\in \mathfrak{F}$$ and the definition of a local formation it follows that $TH/C\in \mathfrak{F}$. 

Let $H^*=\{(h^{-1}K, hC)\mid h\in H\}$. It is an easy calculation to show that $H^*\unlhd (H/K)\rtimes H\overline{T}/C$, $H^*\cap (HT/C)\simeq 1$ and $H^*HT/C=(H/K)\rtimes (HT/C)=\overline{T}$. Now $\overline{T}/(H/K)\simeq \overline{T}/H^*\simeq HT/C\in \mathfrak{F}$. Hence $\overline{T}\in \mathfrak{F}$.

Assume that $H/K$ is abelian. Then $H\leq C$. So $HT/C=T/C\in f(p)$. From $H/K\leq C_{\overline{T}}(H/K)$ it follows that $\overline{T}/C_{\overline{T}}(\overline{U})\in f(p)\cap \mathfrak{F}$ for every chief factor $\overline{U}$ of $\overline{T}$ below $H/K$. It means that $\overline{T}\in \mathfrak{F}$. 

Suppose now that $(H/K)\rtimes (HT/C)\in \mathfrak{F}$. Since $\mathfrak{F}$ is a local formation, $\overline{T}/C_{\overline{T}}(H/K)\in \mathfrak{N}_pf(p)$ for all $p\in \pi(H/K)$ by \cite[Lemma 3.6]{M2023}. Note that $\overline{T}\leq \overline{G}=(H/K)\rtimes G/C$ is a primitive  group of type 1 or 3. If $H/K$ is abelian, then $H/K=C_{\overline{G}}(H/K)$. So $\overline{T}/C_{\overline{T}}(H/K)\simeq HT/C\in \mathfrak{N}_pf(p)$ for all $p\in \pi(H/K)$. If $H/K$ is non-abelian, then $G/C$ complements $C_{\overline{G}}(H/K)=\{(h^{-1}K, hC)\mid h\in H\}$. It means $C_{\overline{G}}(H/K)=C_{\overline{T}}(H/K)$ and $\overline{T}/C_{\overline{T}}(H/K)\simeq TH/C\in \mathfrak{N}_pf(p)$ for all $p\in\pi(H/K)$.
\end{proof}

\titleformat{\section}[block]{\normalfont}{\thesection}{0pt}{}
\renewcommand{\refname}{\hspace{6cm}\MakeUppercase{Bibliography}}

\end{document}